\documentclass[preprint,12pt]{article}
\usepackage{amsfonts}
\usepackage{bbm}
\usepackage[top=1in, bottom=1in, left=1in, right=1in]{geometry}
\usepackage{amsmath,booktabs,ctable,threeparttable}
\usepackage{amssymb,amsfonts,boxedminipage,enumerate}
\usepackage{amsthm}
\usepackage{algorithm}
\usepackage{algorithmic}
\usepackage{epsfig,graphicx,picins,picinpar,subfigure}
\usepackage{multirow}
\usepackage{verbatim}
\usepackage{bm}
\usepackage[numbers,authoryear,round]{natbib}
\usepackage[colorlinks,
linkcolor=blue,
anchorcolor=blue,
citecolor=blue
]{hyperref}

\makeatletter

\newcommand{\abs}[1]{\left\vert#1\right\vert}

\newcommand{\E}[1]{\mathbb{E}#1}
\newcommand{\R}[1]{\mathbb{R}#1}
\newcommand{\U}[1]{\mathbb{U}#1}

\newcommand{\var}[1]{\mathrm{Var}(#1)}

\newcommand{\I}[1]{\mathbb{I}\left\{#1\right\}}
\newcommand{\hmu}[1]{\hat{\mu}#1}
\newcommand{\condhmu}[1]{{\hat{\mu}}_{\mathrm{cmc}}#1}
\newcommand{\csthmu}[1]{{\hat{\mu}}_{\mathrm{sm}}#1}
\newcommand{\simiid}{\stackrel{\mathrm{iid}}\sim}
\newcommand{\mrd}{\,\mathrm{d}}
\newcommand{\trans}{\top}

\newcommand{\Rmnum}[1]{\expandafter\@slowromancap\romannumeral #1@}

\makeatother
\newtheorem{lemma}{Lemma}
\newtheorem{theorem}{Theorem}

\theoremstyle{definition}
\newtheorem{example}{Example}
\newtheorem{remark}{Remark}
\newtheorem{definition}{Definition}

\begin{document}

\title{An integrated quasi-Monte Carlo method for handling high dimensional problems with discontinuities in financial engineering}

\author{Zhijian He\\hezhijian87@gmail.com \and Xiaoqun Wang \\xwang@math.tsinghua.edu.cn}
\maketitle

%\cortext[cor1]{Corresponding author. Tel.: +86 13450361878.}
%\author[label2]{Zhijian He\corref{cor1}}
%\ead{hezhijian87@gmail.com}
%\address[label2]{Lingnan (University) College, Sun Yat-Sen University,  Guangzhou 510275, China}
%
%\address[label1]{Department of Mathematical Sciences, Tsinghua University,  Beijing 100084, China}
%
%\author[label1]{Xiaoqun Wang}
%\ead{xwang@math.tsinghua.edu.cn}

\begin{abstract}
Quasi-Monte Carlo (QMC) method is a useful numerical tool for pricing and hedging of complex financial derivatives. These problems are usually of high dimensionality and discontinuities. The two factors may significantly deteriorate the performance of the QMC method. This paper develops an integrated method that overcomes  the challenges of the high dimensionality and discontinuities concurrently. For this purpose, a smoothing method is proposed to remove the discontinuities for some typical functions arising from financial engineering. To make the smoothing method applicable for more general functions, a new path generation method is designed for simulating the paths of the underlying assets such that the resulting function has the required form. The new path generation method has an additional power to reduce the effective dimension of the target function. Our proposed method caters for a large variety of model specifications,  including  the Black-Scholes, exponential normal inverse Gaussian L\'evy,  and Heston models.  Numerical experiments dealing with these models show that in the QMC setting the proposed smoothing method in combination with the new path generation method can lead to a dramatic variance reduction for pricing exotic options with discontinuous payoffs and for calculating options' Greeks. The investigation on the effective dimension and the related characteristics explains the significant enhancement of the combined procedure.

\smallskip



\end{abstract}

\section{Introduction}
The Monte Carlo (MC) and the quasi-Monte Carlo (QMC) methods are important numerical tools in the problems of pricing and hedging of
complex financial derivatives \citep[][]{Glasserman2004,lecu:2009,lemieux:2009}. The prices of financial derivatives can be expressed as mathematical expectations of their discounted payoffs with respect to the risk-neutral measure. Hedging portfolios are often constructed from the sensitivities (or Greeks) of the financial derivatives. After some suitable transformations, these problems
can be formulated as integrals over the $d$-dimensional unit cube $(0,1)^d$
\begin{equation}\label{If}
	I(h)=\int_{(0,1)^d}h(\bm{u})\mrd\bm{u},
\end{equation}
with the dimension in hundreds or thousands. For instance, the problem of evaluating mortgage-back securities can be formulated as a high-dimensional integral with a dimension $d$ up to 360 in the examples of \cite{Caflisch1997}. In most cases, the integral~\eqref{If} cannot be calculated analytically and has to be approximated by the MC or QMC method.
QMC has the potential to accelerate the convergence rate of MC. The QMC method estimates the integral~\eqref{If} via
\begin{equation}\label{MCappro}
	Q_{N}(h)=\frac{1}{N}\sum_{i=1}^Nh(\bm{u}_i),
\end{equation}
where $\bm{u}_i\in(0,1)^d$ are deterministic and more uniformly distributed points known as \emph{low discrepancy points} instead of plain pseudo-random points used in MC. The well-known Koksma-Hlawka inequality guarantees that the QMC method
yields a deterministic error bound $O(N^{-1}(\log N)^d)$ for functions of finite variation in the sense of Hardy and Krause \citep[][]{Niederreiter1992}. 

High dimensionality and discontinuities are two key factors that may deteriorate the performance of QMC as well as some other numerical methods, e.g., trapezoidal rules and sparse grid quadratures \citep[][]{holt:2011}.
To overcome the challenge of high dimensionality arising from finance, some path generation methods (PGMs) have been proposed to reduce the effective dimension of the payoffs, since the QMC method favors the problems with low effective dimension \citep[][]{Caflisch1997,Wang2003}.  Particularly, \cite{Imai2006} proposed the linear transformation (LT) method that aims at minimizing the effective dimension of the target function. On the other hand, 
\cite{Wang2013} found that PGMs have a strong impact on the discontinuity structure and proposed the orthogonal
transformation method that realigns the discontinuities to be parallel to the coordinate axes (such discontinuities are referred as \emph{QMC-friendly} since with such discontinuities good performance can still be expected for QMC). In this way, the difficulty of discontinuities is  partially overcome. Subsequently, \cite{he:wang:2014}
developed a more general PGM (called the QR method) based on the QR decomposition  of a matrix \citep[][]{Golub2012} that can deal with multiple discontinuity structures by proper realignments. Moreover, it enjoys the effect of dimension reduction. To enhance the effect of dimension reduction further, \cite{imai:tan:2014} proposed to integrate the LT method and the orthogonal
transformation method. 

Although the discontinuities can be realigned to be QMC-friendly by some proper PGMs, discontinuities are still involved
in the resulting function which may  more or less dampen the efficiency of QMC.  One natural question is whether we can remove the discontinuities
completely to improve the smoothness of the function. To this end, \cite{wang:2015} proposed a smoothing method to remove the discontinuities.  The smoothing method works with simple discontinuous functions of the form $f(\bm{u})\I{\Gamma_1<u_1<\Gamma_2}$, where $f(\cdot)$ is a smooth  function, $\I{\cdot}$ is an indicator function, $u_1$ is the first entry of the vector $\bm{u}$, and $\Gamma_1$ and $\Gamma_2$ are constants.
However, this restriction rules out discontinuous payoffs of many commonly traded options, e.g., binary Asian options and barrier options. Moreover, a specific PGM is required to translate the payoffs into this form as shown in the examples of \cite{wang:2015}.  Some smoothing methods are widely used in  MC, but for different purposes \citep[][]{Glasserman2004,Liu2011}.

For high-dimensional problems with discontinuous functions, special methods are required to ensure the faster convergence of QMC. Our aim is to develop new methods to overcome the challenges of high dimensionality and discontinuities in financial engineering. For this purpose, we propose a two-step procedure that generalizes the procedure in \cite{wang:2015} in both steps. In the first step, a good PGM is designed to transform the function to a form such that the structure of the discontinuity is simplified and the effective dimension is reduced. In the second step, a smoothing method is proposed to remove the discontinuities completely. The two-step  procedure has the advantage of
removing the discontinuities and reducing the effective dimension. It can be applied to the pricing and hedging of  commonly traded options with discontinuous payoffs, such as binary Asian options and barrier options. Comparing to the method proposed by \cite{wang:2015}, our  method has a wider scope that tailors many classes of contracts and more general models including some exponential  L\'evy models \citep[see, e.g.,][]{tank:2003} and the Heston model \citep[][]{heston:1993}.

The remainder of this paper is organized as follows. 
In Section \ref{sec:CST}, a new smoothing method is developed.
In Section \ref{sec:cstFP}, we show how to apply the proposed smoothing method to problems with typical payoffs under a general framework.
We then propose a dimension reduction method adapted to the smoothing method
for some problems. In Section \ref{sec:numerical}, extensive numerical experiments are performed on pricing some exotic
options and calculating their Greeks under the Black-Scholes model and the exponential normal inverse Gaussian (NIG) model. In Section~\ref{sec:hestonModel}, we discuss the generalization of the proposed method to the Heston framework. Conclusions
are presented in Section \ref{SectionCon}. The concept of effective dimension is deferred to the appendix. 

We conclude this section by  citing some related works based on the Fourier transform in computational finance. As an alternative to MC and QMC, the Fourier transform is employed to the pricing and  hedging of options \citep[][]{ball:2015,fang:2008,fusai:2016}.
In our numerical experiments, we focus on comparing the proposed  method with some existing QMC methods to evaluate the quality of our strategy in the QMC literature. The comparisons with other branches of alternatives are interesting topics for further research.

\section{A New Smoothing Method}\label{sec:CST}
\subsection{Problem Formulation}
Consider the problem of pricing or hedging a path-dependent financial derivative based on asset prices in discrete times $\bm{S}:=(S_1,\dots,S_m)^{\trans}$, where $S_i:=S(t_i)$ denotes the price at the time $t_i$ and $m$ is the number of time steps. For
simplicity, we assume that the asset prices are observed at equally spaced
times, i.e., $t_i=i\Delta t$, where $\Delta t=T/m$ and $T$ is the maturity of the financial derivative. Under the risk-neutral measure, the price and the sensitivities of the financial derivative can often
be expressed as an expectation $\E[g(\bm{S})]$ for a real function $g(\cdot)$ over $\R^m$. Many functions in the pricing and hedging of financial derivatives  can be expressed in the form
\begin{equation}\label{Payoff}
g(\bm{S})=f(\bm{S})\I{q(\bm{S})>0}.
\end{equation}
The two functions $f(\bm{S})$ and $q(\bm{S})$ are usually differentiable almost everywhere. For pricing financial options, the factor $f(\bm{S})$
determines the magnitude of the payoff and $q(\bm{S})>0$ gives the payout condition. For calculating Greeks by the pathwise method \citep[][]{Glasserman2004},
the target function often involves an indicator function as in \eqref{Payoff} even though the underlying payoff is continuous. In this paper, we focus on discontinuous functions of the form~\eqref{Payoff}; see Section~\ref{sec:numerical} for some examples. 

We assume that under the risk-neutral measure the asset prices $\bm{S}$ can be generated by a uniform variate $\bm{u}:=(u_1,\dots,u_d)^\top$ in the unit cube $(0,1)^d$. This implies that $S_i$ can be viewed as a function of $\bm{u}$, say, $F_i(\bm{u})$. Let $F(\bm u)=(F_1(\bm u),\dots,F_m(\bm u))^\top$. Such a mapping $F(\bm u):(0,1)^d\to \R^m$ corresponds to a PGM of the underlying asset process.  Note that the nominal dimension of the mapping is $d$.  It may differ from the number of  time steps $m$ in some situations. As we will see, $d=m$ for the Black-Scholes model, while $d=2m$ for the Heston model.
Our problem is to estimate the expectation
\begin{equation}\label{problem}
	\mu=\E[g(\bm{S})]=\E[g(F(\bm{u}))],\ \bm{u}\sim  \U((0,1)^d).
\end{equation}

It is important to note that the mapping $F(\bm{u})$ in \eqref{problem} is not unique as long as it follows the law of $\bm S$. The mapping $F(\bm{u})$ does not affect the efficiency of the MC method since it does not change the variance of the integrand. However, it may have a significant impact on the performance of the QMC method since it could change the nature of any discontinuity and the effective dimension of the function. How  to design properly a mapping (or PGM) for QMC has become one of the most urgent tasks in this area.

\subsection{The Variable Push-Out Method}
Now we develop a new smoothing method aimed at removing the discontinuities completely and thus improving the smoothness of the target function~\eqref{Payoff}. By taking a transformation $\bm{S}=F(\bm{u})$ in \eqref{Payoff}, we obtain
\begin{equation}\label{Payofftrans}
h(\bm{u}):=g(F(\bm{u}))=f(F(\bm{u}))\I{q(F(\bm{u}))>0}.
\end{equation} 
We introduce the definition of \emph{variable separation condition} which is required for the setup of the proposed smoothing method (suppose that $d>1$).

\begin{definition}\label{Assumption}
	
	The indicator function in \eqref{Payofftrans} is said to satisfy the variable separation condition if there exist two functions $\Gamma_1(\cdot)$ and $\Gamma_2(\cdot)$ depending on $\bm{u}_{2{:}d}=(u_2,\dots,u_d)^{\trans}$ and  $0\le\Gamma_1(\bm{u}_{2{:}d})\le\Gamma_2(\bm{u}_{2{:}d})\le1$ such that
	$\{q(F(\bm{u}))>0\}$  is equivalent to $$\{\Gamma_1(\bm{u}_{2{:}d})<u_1<\Gamma_2(\bm{u}_{2{:}d})\}.$$
\end{definition}

We sometimes abbreviate $\Gamma_1:=\Gamma_1(\bm{u}_{2{:}d})$ and $\Gamma_2:=\Gamma_2(\bm{u}_{2{:}d})$, and assume that $\Gamma_1$ and $\Gamma_2$ are continuous. To carry out our smoothing method, we assume that the indicator function in \eqref{Payofftrans} satisfies the variable separation condition, implying that the function $h(\bm{u})$ in \eqref{Payofftrans} can be written as
\begin{equation}
\label{specialForm}
h(\bm{u})=f(F(\bm{u}))\I{\Gamma_1(\bm{u}_{2{:}d})<u_1<\Gamma_2(\bm{u}_{2{:}d})}.
\end{equation}
The variable separation condition seems restrictive as it depends on both the model of the price dynamics (which determines the mapping $F$) and the payoff (which determines the function $q(\cdot)$ in \eqref{Payofftrans}).
In the next section, we show that the variable separation condition is a naturally occurring condition for options that are frequently traded
in financial markets under a general framework.

The MC approximation of $\mu=\E[h(\bm{u})]$ is given by
\begin{equation}\label{MCestimate}
	\hmu = \frac{1}{N} \sum_{i=1}^N h(\bm{u}_i),\ \bm{u}_i\simiid \U((0,1)^d).
\end{equation}
Given $\bm{v}\in(0,1)^{d-1}$, denote the following conditional expectation as
\begin{equation}\label{eq:condE}
	c(\bm{v}):=\E[h(\bm{u})|\bm{u}_{2{:}d}=\bm{v}].
\end{equation}
The law of total expectation admits
$\mu = \E[c(\bm{u}_{2{:}d})].$
Assume that $c(\bm{v})$ can be calculated analytically for each $\bm{v}\in (0,1)^{d-1}$. This leads to a conditional MC estimate of $\mu$:
\begin{equation}\label{CMCestimate}
	\condhmu = \frac{1}{N} \sum_{i=1}^N c(\bm{v}_i),\ \bm{v}_i\simiid \U((0,1)^{d-1}).
\end{equation}
It is well-known that the conditional MC estimator \eqref{CMCestimate} has a variance no larger than the plain MC
estimator \eqref{MCestimate}. This is justified by  the elementary equality $\var{h(\bm{u})}=\var{c(\bm{u}_{2{:}d})}+\E[\var{h(\bm{u})|\bm{u}_{2{:}d}}].$
The main requirement for conditional MC is that we must be able to compute the conditional expectation $c(\bm{v})$ in \eqref{eq:condE} analytically, which is hard in many cases. In such cases, the conditional MC and its QMC version can be difficult to use. 

To remedy this, we generalize the smoothing
method of \cite{wang:2015} to remove the discontinuities in the function $h(\bm{u})$ given in \eqref{specialForm}. By taking the transformation (for a fixed $\bm{u}_{2{:}d}\in(0,1)^{d-1}$)
\begin{equation}\label{eq:transformation}
	u_1=(\Gamma_2(\bm{u}_{2{:}d})-\Gamma_1(\bm{u}_{2{:}d}))\tilde{u}_1+\Gamma_1(\bm{u}_{2{:}d})=:Z(\tilde{u}_1,\bm{u}_{2{:}d}),
\end{equation}
we have
\begin{align}
	\mu &= \int_{(0,1)^{d-1}}\left(\int_0^1 f(F(\bm{u}))\I{\Gamma_1\le u_1\le \Gamma_2}\mrd u_1\right)\mrd \bm{u}_{2{:}d}\nonumber\\
	&=\int_{(0,1)^{d-1}}\left(\int_0^1 (\Gamma_2-\Gamma_1)f(F(Z(\tilde{u}_1,\bm{u}_{2{:}d}),\bm{u}_{2{:}d}))\I{0\le \tilde{u}_1\le 1}\mrd \tilde{u}_1\right)\mrd \bm{u}_{2{:}d}\notag\\
	&=\int_{(0,1)^{d}} (\Gamma_2-\Gamma_1)f(F(Z(\bm{u}),\bm{u}_{2{:}d}))\mrd \bm{u}.\label{conditioning}
\end{align}
Now let
\begin{equation}\label{smoothedFn}
	\tilde{h}(\bm{u})=(\Gamma_2-\Gamma_1)f(F(\tilde{\bm{u}})),
\end{equation}
where $\tilde{\bm{u}}=(\Gamma_1+(\Gamma_2-\Gamma_1)u_1,\bm{u}_{2{:}d})^{\trans}$,
it follows from \eqref{conditioning} that $\mu=\E[\tilde{h}(\bm{u})]$.
So one can estimate $\mu$ via
\begin{equation}\label{CSTestimate}
	\csthmu = \frac{1}{N} \sum_{i=1}^N \tilde{h}(\bm{u}_i).
\end{equation}

We have shown that by applying the variables transformation in \eqref{eq:transformation}, the variables in the vector $\bm{u}_{2{:}d}$ seem to be ``pushed out" from the indicator function (while the conditional MC integrates out the variables from the indicator function). So we refer to this the \emph{variable push-out} (VPO) smoothing method, and call $\csthmu$ the smoothed estimate.
\cite{Achtsis2013} used a similar idea to price barrier options under the LT method. Their motivation was to make the
sampling scheme compatible with the LT method. We generalize their idea to more general functions of the form \eqref{specialForm}.
Our motivation is to  smooth the integrand for improving the efficiency of QMC.
The following theorem generalizes the results in \cite{wang:2015}.
\begin{theorem}\label{ThmVarReduction}
	Suppose that $h(\bm{u})$ and $\tilde{h}(\bm{u})$ are given in \eqref{specialForm} and \eqref{smoothedFn}, respectively, where $\Gamma_1,\Gamma_2$, $F_i$ and $f$ are continuous functions, and $\bm{u}\sim \U((0,1)^d)$. Then $\tilde{h}(\bm{u})$ has the following properties:
	\begin{enumerate}[(1)]
		\item Continuity: $\tilde{h}(\bm{u})$ is a continuous function for $\bm{u}\in(0,1)^d$;
		\item Unbiasedness: $\E[\tilde{h}(\bm{u})]=\E[h(\bm{u})]$;
		\item Variance reduction: $\var{\tilde{h}(\bm{u})}\le c\var{h(\bm{u})}$,
		where 
		\begin{equation}\label{eq:vrf}
			c=\sup_{\bm{v}\in (0,1)^{d-1}}(\Gamma_2(\bm{v})-\Gamma_1(\bm{v})).
		\end{equation}
	\end{enumerate}
\end{theorem}

\begin{proof}
The continuity is due to the fact that $h_1,h_2$ and $f$ are continuous functions. The unbiasedness is directly obtained from  \eqref{conditioning}. 

The remaining task is to prove that $E[\tilde{h}(\bm{u})^2]\leq c\E[h(\bm{u})^2]$. Since $0\le \Gamma_2-\Gamma_1\le c$, where $c$ is given by \eqref{eq:vrf}, we have

\begin{align*}
	\E[\tilde{h}(\bm{u})^2]  &= \int_{(0,1)^{d}}(\Gamma_2-\Gamma_1)^2f(F(Z(\bm{u})),\bm{u}_{2{:}d})^2\mrd \bm{u}\\
	&\le c\int_{(0,1)^{d-1}}\int_0^1(\Gamma_2-\Gamma_1)f(F(Z(u_1,\bm{u}_{2{:}d}),\bm{u}_{2{:}d})^2\mrd u_1\mrd \bm{u}_{2{:}d}\\
	&=c\int_{(0,1)^{d-1}}\int_{\Gamma_1}^{\Gamma_2}f(F(u_1,\bm{u}_{2{:}d}))^2\mrd u_1\mrd \bm{u}_{2{:}d}=c\E[h(\bm{u})^2].
\end{align*}
Notice that $c\in[0,1]$. By the unbiasedness of $\tilde{h}(\bm{u})$ and the inequality above, we have
\begin{align*}
\var{\tilde{h}(\bm{u})}&=\E[\tilde{h}(\bm{u})^2]-(\E[h(\bm{u})])^2\\&\le c\E[h(\bm{u})^2]-c(\E[h(\bm{u})])^2=c\var{h(\bm{u})},
\end{align*} 
which completes the proof.
\end{proof}

Theorem \ref{ThmVarReduction} guarantees that the smoothed estimate \eqref{CSTestimate} is unbiased and has a variance no larger than
that of the crude MC. \cite{wang:2015} mainly focused on the special case $0\leq \Gamma_1(\bm{u}_{2{:}d})\equiv a<\Gamma_2(\bm{u}_{2{:}d})\equiv b\leq 1$ so that $c=b-a$. The VPO method thus reduces the variance  by at least a  factor $1/c$  compared to crude MC. One can get a great variance reduction when $a$ and $b$ are very close together, but many problems in finance (e.g., arithmetic Asian options) cannot lead to such a trivial case. We thus focus on the general case \eqref{specialForm}.

If the function $F(\bm{u})$ depends only on $\bm{u}_{2{:}d}$, the VPO smoothing method and the conditional MC lead to the same estimate since the conditional expectation in \eqref{eq:condE} has the closed form solution. In general, although the VPO smoothing method may yield a larger variance than the conditional MC, its advantage is that the transformed
integrand $\tilde{h}(\bm{u})$ can be obtained directly without computing any conditional expectations.  

Applying the VPO smoothing method in QMC is straightforward using the same form of the approximation \eqref{CSTestimate}. The unbiasedness is preserved in the context of randomized QMC \citep[][]{lecu:2002}. The nice property of the estimate is that the
function $\tilde{h}(\bm{u})$ does not involve any discontinuities. It is known that smoothness is a key factor affecting the performance of QMC. Better smoothness may yield a faster convergence rate of the QMC estimate.
This was confirmed by \cite{owen1997a}, who proved that the variance of randomized QMC is $O(N^{-3}(\log N)^{d-1})$ for smooth
integrands. That rate is much faster than the rate found for discontinuous functions in \cite{he:wang:2015}. We thus expect that the improved smoothness can increase the efficiency of QMC.

\begin{remark}\label{rm1}
	If $\{q(F(\bm{u}))>0\}$ in \eqref{Payofftrans} is equivalent to
	$\{u_1\leq \Gamma_1\}\cup\{u_1\geq \Gamma_2\}$, the VPO smoothing method is also applicable by recognizing that
	\begin{align}\label{Ext}
	h(\bm{u})&=f(F(\bm{u}))\I{u_1\leq \Gamma_1\}\cup\{u_1\geq\Gamma_2}\notag\\
	&=
	f(F(\bm{u}))-f(F(\bm{u}))\I{\Gamma_1< u_1< \Gamma_2}.
	\end{align}
	Notice that the function $f(F(\bm{u}))$ in  \eqref{Ext} is continuous. We just need to handle the second term
	in \eqref{Ext} by using the VPO smoothing method, resulting in a smooth function $\tilde{h}(\bm{u})=f(F(\bm{u}))-(\Gamma_2-\Gamma_1)f(F(\tilde{\bm{u}})).$
\end{remark}

\section{A Two-Step Procedure in Computational Finance}\label{sec:cstFP}
In this section, we focus on problems of pricing and hedging of financial derivatives in which the dynamic of the asset price $S(t)$ follows an exponential L\'evy model, defined by
\begin{equation}
S(t)=S_0\exp\{L(t)\},
\end{equation}
where $L(t)$ is a L\'evy process, a stochastic process with $L(0)=0$ and independent and identically distributed (i.i.d.) increments. 
In the discrete framework, one can write that
\begin{equation}\label{eq:pgm}
S_i = S_0\exp(x_1+\dots+x_i),
\end{equation}
where $x_i=L(t_i)-L(t_{i-1})$ are i.i.d. variables. Let $\bm x=(x_1,\dots,x_m)^\top$, and denote $\phi(\cdot)$ as the cumulative distribution function (CDF) of $x_i$. Assume  that $\phi$ is continuous. Assume further that $x_i$
can be generated by the inverse method, i.e., $x_i=\phi^{-1}(u_i)$, where $u_i\simiid \U{(0,1)}$.  It is easy to see that the Black-Scholes satisfies the two assumptions; see Section~\ref{sec:discu} for further discussions.  This gives 
\begin{equation}\label{eq:mapping}
S_i = F_i(\bm{u}):=S_0\exp\{\phi^{-1}(u_1)+\dots+\phi^{-1}(u_i)\},
\end{equation}  
where $\bm{u}=(u_1,\dots,u_d)^\top$ and $d=m$ in this setting.

\subsection{Verifying the Variable Separation Condition: Three Typical Cases}\label{sec:varf}
Before using the VPO method in the smoothing step, we need to verify the variable separation condition for the target function \eqref{Payoff}. We consider three
different classes of $q(\cdot)$ involved in \eqref{Payoff}:
\begin{itemize}
	\item component function $q_{\mathrm{C}}(\bm S)=S_j-\kappa$ for some $j\in\{1,\dots,m\}$,	
	\item extreme function $q_{\mathrm{E}}(\bm S)=\min(S_1,\dots,S_m)-\kappa$ or $q_{\mathrm{E}}(\bm S)=\kappa-\max(S_1,\dots,S_m)$,
	\item average function $q_{\mathrm{A}}(\bm S)=S_{\mathrm{A}}-\kappa$, where $S_{\mathrm{A}}:= (1/m)\sum_{i=1}^m S_i$,
\end{itemize}
where $\kappa$ is a constant. The above functions appear in
payoffs of commonly traded options, which are also studied in \cite{tong:liu:2016}. For example, \cite{glas:1999} studied a path-dependent option with discounted payoff $\max(S_{\mathrm{A}}-K,0)\I{S_m>\kappa}$, where $K$ is a constant. It can be translated into  the form \eqref{Payoff} with $f(\bm S)=\max(S_{\mathrm{A}}-K,0)$ and $q(\bm S)=q_{\mathrm{C}}(\bm S)$. A down-and-out barrier call option
with (undiscounted) payoff $\max(S_m-K,0)\I{\min(S_1,\dots,S_m)>\kappa}$ also fits into the form \eqref{Payoff} with $q(\bm S)=q_{\mathrm{E}}(\bm S)$.
A binary Asian option with (undiscounted) payoff $\I{S_\mathrm{A}>\kappa}$ has  the form~\eqref{Payoff} with  $q(\bm S)=q_{\mathrm{A}}(\bm S)$. 

For the component function $q_\mathrm{C}(\bm S)=S_j-\kappa$, using \eqref{eq:mapping} gives
\begin{align*}
\{q_\mathrm{C}(\bm S)>0\}&\Leftrightarrow\left\lbrace S_0\exp\left(\sum_{i=1}^jx_i\right)>\kappa\right\rbrace\Leftrightarrow \left\lbrace u_1>\gamma_j(\bm u_{2{:}d};\kappa)\right\rbrace,
\end{align*}
where 
\begin{equation}\label{eq:cf}
\gamma_j(\bm u_{2{:}d};\kappa):=\phi\left(\log(\kappa/S_0)-\sum_{i=2}^jx_i\right)\leq 1.
\end{equation}
The variable separation condition is thus verified by setting $\Gamma_1=\gamma_j(\bm u_{2{:}d};\kappa)$ and $\Gamma_2=1$. 

For the extreme function $q_\mathrm{E}(\bm S)=\min(S_1,\dots,S_m)-\kappa$, it is easy to find that
\begin{align*}
\{q_\mathrm{E}(\bm S)>0\}&\Leftrightarrow \bigcap_{j=1}^m \{S_j-\kappa>0\}\Leftrightarrow \bigcap_{j=1}^m \{u_1>\gamma_j(\bm u_{2{:}d};\kappa)\}\\
&\Leftrightarrow \left\lbrace u_1>\max_{j=1,\dots,m}\gamma_j(\bm u_{2{:}d};\kappa)\right\rbrace.
\end{align*}
So the variable separation condition  holds with $\Gamma_1=\max_{j=1,\dots,m}\gamma_j(\bm u_{2{:}d};\kappa)$ and $\Gamma_2=1$. Similar analysis applies to the case of $q_{\mathrm{E}}(\bm S)=\kappa-\max(S_1,\dots,S_m)$.

We next consider the average function $q_{\mathrm{A}}(\bm S)=S_{\mathrm{A}}-\kappa$. It follows from~\eqref{eq:mapping} that
\begin{align*}
S_{\mathrm{A}} =\frac 1m\sum_{i=1}^m S_i=\frac{\exp[\phi^{-1}(u_1)]}{m}\sum_{i=1}^m S_0\exp\left(\sum_{j=2}^ix_j\right).
\end{align*}
We thus arrive at the  equivalence
\begin{equation*}
\{S_{\mathrm{A}}-\kappa>0\}\Leftrightarrow
\{\gamma(\bm{u}_{2{:}d})<u_1<1\},\label{EquelSA}
\end{equation*}
where
\begin{equation}\label{eq:lbsa}
\gamma(\bm{u}_{2{:}d}):=\phi\left(\log(\kappa m)-\log\left(\sum_{i=1}^m S_0\exp\left(\sum_{j=2}^ix_j\right)\right)\right)\le 1.
\end{equation}
This implies that the variable
separation condition is satisfied by setting $\Gamma_1=\gamma(\bm u_{2{:}d})$ and $\Gamma_2=1$. Since the CDF $\phi$ is assumed to be continuous, $\Gamma_1$ and $\Gamma_2$ are continuous for the three cases. By Theorem~\ref{ThmVarReduction}, the VPO method yields a smoothed and unbiased estimate with reduced variance.

It is easy to see that the variable
separation condition still holds for functions involving multiple indicators of the form
\begin{equation*}\label{multipPayoff}
g(\bm{S})=f(\bm{S})\prod_{j=1}^{J}\I{q_j(\bm{S})>0},
\end{equation*}
where $q_j$ belong to the three cases above and $J>1$. Let $\kappa_j=\kappa$ for $j=1,\dots,m-1$, $\kappa_d=\max(K,\kappa)$. The payoff of the down-and-out barrier option can be rewritten as
\begin{align}
g(\bm{S})&=(S_m-K)\prod_{j=1}^m\I{S_j>\kappa_j}\notag\\
&=(S_m-K)\prod_{j=1}^m\I{u_1>\gamma_j(\bm{u}_{2{:}d};\kappa_j)}\notag\\
&=(S_m-K)\I{u_1>\max_{j=1,\dots,m}\gamma_j(\bm u_{2{:}d};\kappa_j)},\label{eq:barrier}
\end{align}
where $\gamma_j(\bm u_{2{:}d};\kappa_j)$ is given by \eqref{eq:cf}. From this point of view, the variable
separation condition holds if we take $q(\bm S)=S_m-K$ to translate into the form \eqref{Payoff}. By doing so, the function $q(\bm S)$ is smoother than the function $\max(S_m-K,0)$.

\begin{remark}\label{rem:genways}
	The way to generate $\bm{x}$ is not unique.	To verify the variable separation condition for the three classes of functions, it only requires that $x_1$ can be expressed as an invertible function of $u_1$ and $x_j$ depends on $\bm u_{2{:}d}$ for any $j\in\{2,\dots,m\}$. This leaves room to exploit other ways to generate $\bm{x}$ from $\bm u$ such that the variable separation condition is satisfied. 
\end{remark}

\subsection{A New Dimension Reduction Method for Gaussian Cases}\label{SectionPGM}
Section~\ref{sec:varf} shows that under the exponential L\'evy framework (satisfying some conditions), the VPO method is applicable for the function \eqref{Payoff} with three different forms of $q(\cdot)$.
The analysis relies on the mapping \eqref{eq:mapping} that provides a usual way to generate the asset prices. Actually there are many other ways to generate the asset prices such that the variable separation condition  is also satisfied. To be more precisely, we restrict our attention to Gaussian cases in which
$x_i$ in \eqref{eq:pgm} are i.i.d. normal variables $N(a,b^2)$ for $a\in\R$ and $b>0$. This implies that $\bm{x}\sim N(a\bm{1},b^2\bm I_d)$. To generate $\bm{x}$, we usually take $\bm{x}=a\bm 1+b\bm{z}$, where $\bm{z}\sim N(\bm 0,\bm{I}_d)$. In doing so, $\bm S$ can be written as a function of $\bm z$, denoted by $\mathcal{S}(\bm{z})$. 
We thus arrive at the following equalities
\begin{equation}\label{eq:newproblem}
\mu=\E[g(\bm{S})]=\E{[g(\mathcal{S}(\bm{z}))]}=\E{[g(\mathcal{S}(\bm{Uz}))]},
\end{equation}
where $\bm{U}$ is an arbitrary orthogonal matrix. The last equality in \eqref{eq:newproblem} is due to the fact that $\bm{x}=a\bm 1+b\bm{Uz}\sim N(a\bm{1},b^2\bm I_d)$ holds for any orthogonal matrix~$\bm{U}$.

Note that different choices of the  orthogonal matrix $\bm{U}$ lead to different mappings $F(\bm{u})=\mathcal{S}(\bm{U}\Phi^{-1}(\bm u))$ after using the inverse transformation $\bm{z}=\Phi^{-1}(\bm u)$, where $\Phi$ is the CDF of the standard normal. The mapping \eqref{eq:mapping} corresponds to the simple case  $\bm{U}=\bm I_d$. We should note that not all orthogonal matrices $\bm{U}$ can satisfy the variable separation condition. The next lemma shows that there exists a class of orthogonal matrices $\bm U$ that makes the VPO smoothing method applicable for our problems. 

\begin{lemma}\label{lam:UU}
Suppose that $F(\bm{u})=\mathcal{S}(\bm{U}\Phi^{-1}(\bm u))$. If the orthogonal matrix $\bm{U}$ has the form
\begin{equation}\label{UMatrix}
\bm{U}=
\left[\begin{array}{cccc}
1&0&\cdots&0\\
0\\
\vdots& &\bm{V}\\
0\\
\end{array}\right],
\end{equation}
where $\bm{V}$ is an arbitrary $(d-1)\times (d-1)$ orthogonal matrix, the  variable separation condition still holds for  the function \eqref{Payofftrans} with the three different forms of $q(\cdot)$ given in Section~\ref{sec:varf}.    
\end{lemma}
\begin{proof}
If $\bm{x}=a\bm 1+b\bm{Uz}$, where $\bm{U}$ is given in \eqref{UMatrix}, then $x_1=\phi^{-1}(u_1):=a+b\Phi^{-1}(u_1)$ independently of $\bm{U}$, and $x_2,\dots,x_d$ depend on $\bm u_{2{:}d}$. The  variable separation condition can be verified following the analysis in Section~\ref{sec:varf}. 
\end{proof}

Lemma~\ref{lam:UU} admits that we  are free to select the orthogonal matrix $\bm{V}$ involved in \eqref{UMatrix}. If we choose $\bm{V}$ naively, the resulting estimate \eqref{smoothedFn}  may have high effective dimension though it is continuous. An effective implementation of QMC  must
concurrently take into consideration the whole function. To this end, we can make full use of the unspecific orthogonal
matrix $\bm{V}$. Motivated by the QR method \citep[][]{he:wang:2014}, we propose a modified QR method (MQR) to determine the matrix $\bm{V}$ (thus the matrix $\bm U$) aiming at reducing the dimension of the
target function. 

To motivate a good choice of $\bm{V}$, we now assume that the target function $g(\mathcal{S}(\bm{z}))$ has the form
\begin{equation}\label{PayoffQR}
G(\bm{w}_1^{\trans}\bm{z},\dots,\bm{w}_r^{\trans}\bm{z})=:G(\bm{W}^{\trans}\bm{z}),\ \bm{z}\sim N(\bm{0},\bm{I}_d),
\end{equation}
where $\bm{w}_i\in\R^d$, $\bm{W}:=[\bm{w}_1,\dots,\bm{w}_r]\in \R^{d\times r}$. The QR method is also based on the form \eqref{PayoffQR}; see \cite{he:wang:2014} for details.
Suppose $\bm U$ has the form \eqref{UMatrix}; then we have
\begin{equation}\label{MQReq}
\bm{W}^{\trans}\bm{Uz}=z_1\bm{W}_1^\top+\bm{W}_{-1}^\top\bm{V}\bm{z}_{2{:}d},
\end{equation}
where $\bm{W}_1$ is the first row of $\bm{W}$, and $\bm{W}_{-1}$ is the remaining $d-1$ row of $\bm{W}$,  and
$\bm{z}_{2{:}d}=(z_2,\dots,z_d)^{\trans}$. We perform a QR decomposition on $\bm{W}_{-1}$, resulting in
\begin{equation}\label{QR}
\bm{W}_{-1}=\bm{Q}\bm{R},
\end{equation}
where $\bm{Q}\in \R^{(d-1)\times (d-1)}$ is an orthogonal matrix and $\bm{R}\in \R^{(d-1)\times r}$ is
an upper triangular matrix. Let $\bm{V}=\bm{Q}$, it follows from \eqref{MQReq} that
\begin{equation*}%\label{WQR}
\bm{W}^{\trans}\bm{Uz}=z_1\bm{W}_1^\top+(\bm{Q}\bm{R})^{\trans}\bm{Q}\bm{z}_{2{:}d}
=z_1\bm{\alpha}+\bm{R}^{\trans}\bm{z}_{2{:}d}=:z_1\bm{\alpha}+\bm{L}\bm{z}_{2{:}d},
\end{equation*}
where $\bm{L}:=\bm{R}^{\trans}\in \R^{r\times (d-1)}$ is a lower triangular matrix. We  obtain that
\begin{equation*}
G(\bm{W}^{\trans}\bm{Uz})=G(\underbrace{\alpha_1z_1+\ell_{11}z_2}_{\text{two variables}},
\underbrace{\alpha_2z_1+\ell_{21}z_2+\ell_{22}z_3}_{\text{three variables}},\dots,\underbrace{\alpha_rz_1+\ell_{r1}z_2+\dots+
	\ell_{rr}z_{r+1}}_{\text{$r+1$ variables}}),
\end{equation*}
which depends on $r+1$ variables. Thus we achieve a dimension reduction when $r<d-1$. We summarize the results in the following theorem.
\begin{theorem}\label{MQR}
	Assume that the function $g(\mathcal{S}(\bm{z}))$ has the form $G(\bm{W}^{\trans}\bm{z})$ given in \eqref{PayoffQR},
	where $\mathrm{rank}(\bm{W})=r< d$. Let $\bm{U}$ be an orthogonal matrix of the form \eqref{UMatrix} in which $\bm{V}$
	equals $\bm{Q}$ determined by the QR decomposition \eqref{QR}. Then $g(\mathcal{S}(\bm{Uz}))$
	is  changed to the form
	\begin{align}
	G(\alpha_1z_1+\ell_{11}z_2,
	\alpha_2z_1+\ell_{21}z_2+\ell_{22}z_3,\dots,\alpha_rz_1+\ell_{r1}z_2+\dots+
	\ell_{rr}z_{r+1}),\label{Newf}
	\end{align}
	where $\alpha_i$, $\ell_{ij}$ are constants. 
\end{theorem}

Theorem~\ref{MQR} also holds if
$\text{rank}(\bm{W})=d$, but the last two arguments in \eqref{Newf} depend on $d$ variables. The MQR method is very attractive
when $r\ll d$ since the dimension can be reduced significantly.
There are at least three benefits of using the MQR method:
(a) The VPO smoothing method is applicable.
(b) The MQR method has the ability to handle multiple structures as the QR method.
(c) It has the potential to reduce the effective dimension of the target function. All these aspects are beneficial to QMC.

\begin{remark}\label{rem:matching}
	In practice the target functions rarely confirm  the form $G(\bm{W}^{\trans}\bm{z})$ in \eqref{PayoffQR}. To remedy this, we use the first-order Taylor approximation to get the desired form. For the target function $g(\mathcal{S}(\bm{z}))$ of the form \eqref{Payoff}, one possible way is to take the first-order Taylor approximations of the sub-functions $f(\mathcal{S}(\bm{z}))$ and $q(\mathcal{S}(\bm{z}))$. Let $\mathcal{V}_f(\bm z):=(\frac{\partial f(\mathcal{S}(\bm{z}))}{\partial z_1},\cdots,\frac{\partial f(\mathcal{S}(\bm{z}))}{\partial z_d})^\top$ be the gradient vectors of $f(\mathcal{S}(\bm{z}))$ and similarly for $\mathcal{V}_q(\bm z)$. The  first-order Taylor approximation admits
	\begin{align*}
	f(\mathcal{S}(\bm{z}))&\approx f(\mathcal{S}(\bm{z}_0))+\mathcal{V}_f^\top(\bm z_0)(\bm z-\bm z_0),\text{ and}\\q(\mathcal{S}(\bm{z}))&\approx q(\mathcal{S}(\bm{z}_0))+\mathcal{V}_q^\top(\bm z_0)(\bm z-\bm z_0).
	\end{align*}
	By doing so, the target function $g(\mathcal{S}(\bm{z}))$ can be approximated by a function of the required form $G(\mathcal{V}_q^\top(\bm z_0)\bm z,\mathcal{V}_f^\top(\bm z_0)\bm z)$.
	In practice, we take the mean of $\bm z$ as the fixed point $\bm z_0$, i.e., $\bm z_0=\bm{0}$. 
	Therefore, the MQR method is not restricted to functions of the special form $G(\bm{W}^{\trans}\bm{z})$. The only ingredient required for the setup of the MQR method is the weight matrix $\bm{W}$, as required for the QR method. We refer to \cite{he:wang:2014} for the choices of $\bm{W}$ for some typical finance problems. 	
\end{remark}

\subsection{Discussions}\label{sec:discu}
Generally speaking, Section~\ref{SectionPGM} hinges on the assumption that the target function	can be expressed as a function of $\bm{x}$ whose entries are normal random variables. Motivated by \cite{imai:tan:2009}, the MQR method can be extended to general distributions of $\bm{x}$ (the increments of the L\'evy process). More specifically, let's rewrite the target function as $g(\bm{x})$, where $x_i$ are i.i.d. random variables whose CDFs are $\phi(\cdot)$. Using the inverse transformation $x_i=\phi^{-1}(u_i)$, we arrive at 
\begin{equation*}
\mu = \E{[g(\bm x)]}=\E{[g(\phi^{-1}(\bm u))]},
\end{equation*}
where $\phi^{-1}(\bm u):=(\phi^{-1}(u_1),\dots,\phi^{-1}(u_d))^\top$.
Due to the fact that $\bm u=\Phi(\bm{z})\sim \U{(0,1)}^d$, 
the problem can be transformed to the Gaussian case 
\begin{equation*}
\mu =\E{[g(\phi^{-1}(\bm u))]}=\E{[g(\phi^{-1}(\Phi(\bm{z})))]}=\E{[g(\phi^{-1}(\Phi(\bm{Uz})))]},
\end{equation*}
where $\bm{U}$ is an arbitrary orthogonal matrix. 
Thus the MQR method can be applied to the transformed function $g(\phi^{-1}(\Phi(\bm{Uz})))$ to obtain a good $\bm U$ as did in Section~\ref{SectionPGM}.

A key issue to implement the proposed method is that the  increments of the L\'evy process 
can be generated by the inverse method. For most L\'evy processes, however, we only know the density function or the characteristic function
of the increments, and  the inverse CDF $\phi^{-1}$ cannot be expressed analytically. One may thus resort to some numerical inversion methods. When the density function is known explicitly, \cite{imai:tan:2009} advocated the numerical inversion method proposed by \cite{horm:2003}, which is based on Hermite interpolation. The method is fast and could yield accuracy closed to machine precision. \cite{chen:feng:2012} showed how to compute the inverse CDF numerically  by  the Hilbert transform method when the characteristic function is available. Our proposed method can therefore be applied to a more general exponential L\'evy processes, for example, the generalized hyperbolic (GH) process. The GH process encompasses many important models, such as the hyperbolic model and the NIG model. We refer to the monograph \cite{tank:2003} for the details on various L\'evy processes.

Actually, the VPO method has a wider scope than the exponential L\'evy framework presented in this section.  As an extension, Section~\ref{sec:hestonModel} shows how the proposed method works for the Heston model, which is no long an  exponential L\'evy model.

\section{Numerical Experiments}\label{sec:numerical}
\subsection{Black-Scholes Model}
We perform some numerical experiments under the Black-Scholes framework. Under the risk-neutral measure, the asset follows the geometric Brownian motion, a celebrated exponential L\'evy process  
\begin{equation}
	\frac{\mrd S(t)}{S(t)}=r\mrd t+\sigma \mrd B(t),\label{GBM}
\end{equation}
where $r$ is
the riskless interest rate, $\sigma$ is the volatility and $B(t)$
is a standard Brownian motion. Under this framework, the solution of \eqref{GBM} is analytically available $S(t)=S_0\exp\{L(t)\}$, where $L(t)=(r-\sigma^2/2)t+\sigma B(t)$
and $S_0$ is the initial price of the asset. The standard way given by \eqref{eq:pgm}  generates the prices $S_i=S_0\exp(x_1+\dots+x_i)$ via
\begin{equation*}
x_i = L(t_1)-L(t_{i-1})=(r-\sigma^2/2)\Delta t+\sigma (B(t_i)-B(t_{i-1})).
\end{equation*}
Note that the increments of Brownian motion $B(t_i)-B(t_{i-1})\simiid N(0,\Delta t)$. As a result, $x_i\simiid N(a,b^2)$ with $a = (r-\sigma^2/2)\Delta t$ and $b^2=\sigma^2\Delta t$. So the prices can be generated via
\begin{align}
x_i &= a+bz_i,\notag\\
S_i &=S_0\exp\left(\sum_{j=1}^ix_j\right)= S_0\exp\left(ai+b\sum_{j=1}^iz_j\right),\label{eq:SiGBM}
\end{align}
where $z_i\simiid N(0,1)$.
The VPO smoothing method is applicable under the PGM \eqref{eq:SiGBM} for the three typical kinds of problems in  Section~\ref{sec:varf}. However, the VPO smoothing method is generally infeasible when other common PGMs are used, such as the Brownian bridge (BB)
\citep{Caflisch1997}, the principal component analysis (PCA) \citep{Acworth1998} and the QR method proposed by \cite{he:wang:2014}. That is why we need the MQR method. Using the MQR method is straightforward as $x_i$ are normal variables; see Section~\ref{SectionPGM}.	
We examine the effectiveness of the proposed method for the following examples. 

\begin{example}\label{ex:1}
A binary Asian option is an option with the discounted payoff
	\begin{equation}
		g_1(\bm{S})=e^{-rT}\I{S_{\mathrm{A}}>K},\label{Binary}
	\end{equation}
	where $K$ is the strike price. To fit into the form~\eqref{Payoff}, we take $f(\bm{S})=e^{-rT}$ and $q(\bm{S})=S_{\mathrm{A}}-K$. Following the analysis in Section~\ref{sec:varf}, $\Gamma_1=\gamma(\bm{u}_{2{:}d})$ given by \eqref{eq:lbsa} and $\Gamma_2=1$.
\end{example}
\begin{example}\label{ex:2}
	The pathwise estimate for the delta of an arithmetic Asian option with a discounted payoff $e^{-rT}(S_{\mathrm{A}}-K)^+$ is given by
	\begin{equation}\label{Delta}
		g_2(\bm{S})=e^{-rT}\frac{S_{\mathrm{A}}}{S_0}\I{S_{\mathrm{A}}>K}.
	\end{equation}
	The Greek delta is the sensitivity of a financial derivative with respect to the initial price of the underlying asset. To fit into the form~\eqref{Payoff}, we take $f(\bm{S})=e^{-rT}S_{\mathrm{A}}/S_0$ and $q(\bm{S})=S_{\mathrm{A}}-K$. The bounds $\Gamma_1$ and $\Gamma_2$ are the same as in Example~\ref{ex:1}.
\end{example}
\begin{example}\label{ex:3}
	A down-and-out barrier call option is an option with the discounted payoff
	\begin{align}
		g_3(\bm{S})&=e^{-rT}\max(S_m-K,0)\I{\min(S_1,\dots,S_m)>\kappa}\notag\\
		&=e^{-rT}(S_m-K)\prod_{j=1}^{m}\I{S_j>\kappa_j},\label{gamma}
	\end{align}   
	where $\kappa_j=\kappa$ for $j=1,\dots,m-1$, and $\kappa_m=\max(K,\kappa)$. We take $f(\bm{S})=e^{-rT}(S_m-K)$ and $q(\bm{S})=\ \min(S_1-\kappa_1,\dots,S_{m}-\kappa_m)$ to  translate the payoff into the form \eqref{Payoff}. It follows from \eqref{eq:barrier} that $\Gamma_1=\max_{j=1,\dots,m}\gamma_j(\bm u_{2{:}d};\kappa_j)$ and $\Gamma_2=1$.
\end{example}

It is worth pointing out that all target functions in Examples \ref{ex:1}--\ref{ex:3} do not have the  required form for the QR or MQR method
since the arithmetic average of the stock prices $S_{\mathrm{A}}$ is not a function of linear combinations of normal variables. Thus we cannot apply the QR or MQR method directly. As discussed in Remark~\ref{rem:matching}, we take the first-order Taylor approximation for $S_{\mathrm{A}}$. Denote  $\bm{w}_0$ as the gradient vector of  $S_{\mathrm{A}}$ evaluated at $\bm{z}_0=\bm 0$. The arithmetic average $S_{\mathrm{A}}$ can thus be approximated by a function of $\bm{w}_0^\top\bm z$.  For the functions \eqref{Binary} and  \eqref{Delta}, we thus obtain matching functions of the desired form $G(\bm{w}_0^\top\bm z)$. From \eqref{eq:SiGBM}, we find that  $S_i$ can be expressed as a function of $\bm w_i^\top \bm z$, where $\bm{w}_i$ is a $d$-dimensional vector with the first $i$ entries $1$ otherwise $0$. For the function \eqref{gamma}, we thus obtain a matching function of the desired form $G(\bm{w}_1^\top\bm z,\dots,\bm{w}_m^\top\bm z)$.   For Examples~\ref{ex:1} and  \ref{ex:2}, the matrix $\bm{W}$ used for the QR and MQR methods is chosen to be $\bm{W}=\bm{w}_0$ as suggested by the matching
functions, while for Example \ref{ex:3}, we set $\bm{W}=[\bm{w}_m,\bm{w}_{m-1},\dots,\bm{w}_{1}]$ as suggested by \cite{he:wang:2014}. 

We benchmark the relative efficiency of the unsmoothed and smoothed QMC methods to the MC method by computing the variance reduction factor
(VRF). The VRF is the empirical variance of the crude MC estimate divided by the empirical variance of the estimate of
interest, i.e.,
$$\mathrm{VRF}:=\frac{\hat{\sigma}^2_{\mathrm{MC}}}{\hat{\sigma}^2},$$
where $\hat{\sigma}^2_{\mathrm{MC}}$
denotes the sample variance of the crude MC method and $\hat{\sigma}^2$
denotes
the sample variance of the method under consideration.  If there is no substantial difference among the costs of implementing various
QMC methods, the larger is the VRF the more effective is
the underlying QMC-based method.
In our experiments, we use a scrambled version of Sobol' points proposed
by \cite{matousek:1998}, which has lower computational demand than the full scrambling by  \cite{Owen1995}. 
We focus on comparing the performance of the following five methods: 
\begin{itemize}
	\item MC: the plain MC without using the VPO smoothing method,
	\item QMC-I: the same as MC but using low discrepancy points instead of pseudo-random points,
	\item QMC-II: combining QMC-I  with the QR method,
	\item sQMC-I: the VPO method in QMC, i.e., a smoothed version of QMC-I,
	\item sQMC-II: the two-step procedure which incorporates the VPO and MQR methods in QMC.
\end{itemize}
We do not compare the proposed method with other traditional PGMs (e.g., BB and PCA) as \cite{he:wang:2014} showed that the QR method (QMC-I) performs the best among those unsmoothed QMC methods. VRFs reported in all examples are based on a sample
size $N=4096$ and estimated with  $100$ replications. In addition to VRF, we also report the
computational costs (CPU time), where all experiments
are conducted using MATLAB on a PC with
2.6 GHz CPU and 8 GB RAM. 

%The MATLAB code is available at http://www.hezhijian.com/code/code4VPO.zip.

The parameters are  $S_0=100$, $r=0.04$, $\sigma=0.3$, $K=100$, $\kappa = 90$ and $d=m\in\{16,128\}$. Table~\ref{tab:BS} reports the estimates, VRFs and CPU times for the five methods. In
Figure~\ref{fig:1}, we report the sample variance against the computational cost for Example~\ref{ex:1}. The results show the following:

\begin{itemize}
	\item The sQMC-II method has a consistent advantage over other methods in all examples,
	which attains VRFs as high as several thousands or even hundreds of thousands relative to plain MC.  In particular, the sQMC-II method performs much better than the  sQMC-I method, though both methods use the VPO smoothing method (they differ only in whether or not the MQR method is used).
	\item For the barrier option (Example 3), the  sQMC-II method  has small advantage
	over other methods. This is because the payoff \eqref{gamma} involving $m$ discontinuity structures looks more complicated than those of Examples \ref{ex:1} and \ref{ex:2}. In addition, $\Gamma_1=\max_{j=1,\dots,m}\gamma_j(\bm u_{2{:}d};\kappa_j)$ for Example~\ref{ex:3} has some cusps, implying that the associated smoothed function \eqref{smoothedFn} is not as smooth as those in Examples \ref{ex:1} and \ref{ex:2}.
	\item Table~\ref{tab:BS} shows that the computational costs are very close for the QMC methods with/without using the proposed method, which are no more than three times of the cost of plain MC. Figure~\ref{fig:1} shows that the proposed sQMC-II method is still preferable when the computation cost is taken into consideration. The nominal dimension $d$ has a small impact on the performance of the  sQMC-II method.
	\item Surprisingly, the sQMC-I method, which is the smoothed version of QMC-I, often performs much worse than the best unsmoothed QMC method (i.e., the QMC-II method). This indicates that making discontinuities QMC-friendly by the QR method could yield higher accuracy than smoothing the integrand naively.
\end{itemize}

\begin{figure}[ht]
	\centering
	\includegraphics[width = \hsize]{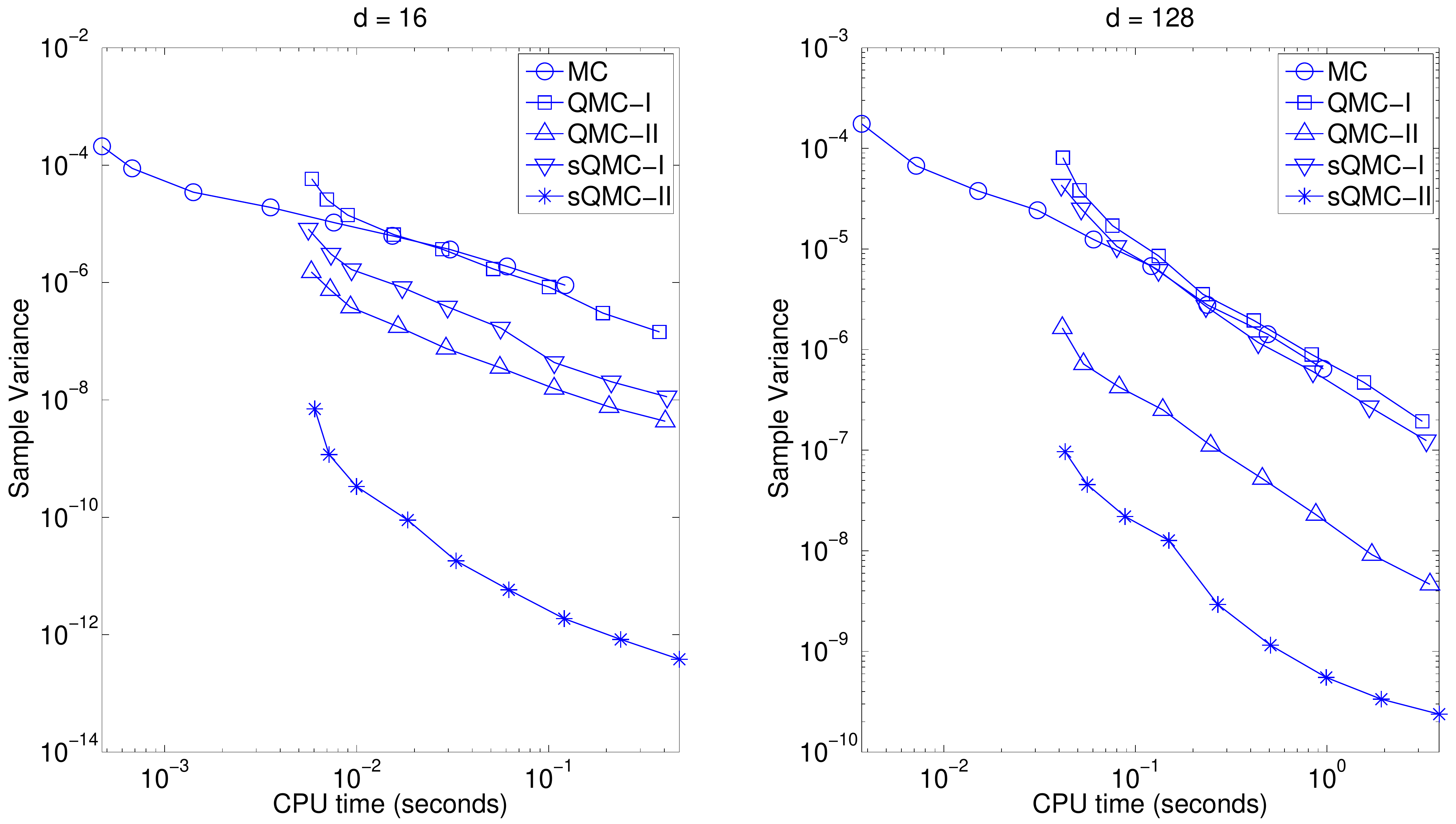}
	\caption{Binary Asian option: pointwise sample variance as a function of CPU time (in seconds) for $d=16$ (left) and $d=128$ (right). The sample sizes are $N=2^i,\ i=10,\dots,18$.\label{fig:1}}
\end{figure}

There is a large difference between the effectiveness of the sQMC-I and sQMC-II methods (both use the VPO smoothing method). A natural question is what leads to the huge diversity of  using MQR or not.
To understand the effect of using MQR, we intentionally compute some effective dimension-related characteristics (such as the truncation variance ratios concentrated on the first variable $R_{\{1\}}$ and on the first two variables $R_{\{1,2\}}$, the degree of additivity $R_{(1)}$, the effective dimension in truncation sense $d_t$, and the mean dimension $d_{\mathrm{ms}}$; see Appendix A for their definitions) for the smoothed integrands. These quantities are closely related to global sensitivity indices. There is some works on application of global sensitivity analysis for assessing QMC efficiency in finance \citep[][]{bian:2015,kuch:shah:2007}.
The numerical results are reported in Table~\ref{tab:efBS}, which are estimated by the QMC method with a large sample size $2^{20}=1048576$.

We observe that the sQMC-II method leads to larger
degree of additivity $R_{(1)}$, much smaller truncation dimension $d_{t}$ and mean dimension $d_{\mathrm{ms}}$ than the sQMC-I method does in all the cases. For Example~\ref{ex:1}, we  observe that $R_{\{1\}}\doteq 99.9\%$ for the sQMC-II method, implying that the resulting functions  are nearly one-dimensional. For
Example \ref{ex:2}, the integrands resulting from the sQMC-II method are nearly two-dimensional since $R_{\{1,2\}}\doteq99.9\%$. These may explain why the sQMC-II method performs much better than the sQMC-I  method and demonstrate the great power of the MQR method in dimension reduction for Examples \ref{ex:1} and \ref{ex:2}. For Example~\ref{ex:3}, the effect of dimension reduction is insignificant. This may explain why the sQMC-II method delivers relatively small VRFs for Example~\ref{ex:3}.

\begin{table}[htbp] 
	\centering
	\caption{Numerical results for the Black-Scholes model}\vspace{0.3cm}
	\begin{tabular}{clrrrrr}
		\toprule
		Cases &
		   & \multicolumn{1}{c}{MC} & \multicolumn{1}{c}{QMC-I} & \multicolumn{1}{c}{QMC-II} & \multicolumn{1}{c}{sQMC-I} & \multicolumn{1}{c}{sQMC-II} \\\midrule
		 \multicolumn{7}{c}{$d=16$} \\
		\multirow{3}[0]{*}{Ex. 1} & Est.  & 0.484812  & 0.484873  & 0.484830  & 0.484741  & 0.484805  \\
		& VRF   & 1     & 3     & 49    & 23    & 59331  \\
		& Time  & 1.5   & 9.2   & 8.9   & 9.1   & 9.5  \\\hline
		\multirow{3}[0]{*}{Ex. 2} & Est.  & 0.566081  & 0.566004  & 0.565950  & 0.565886  & 0.565921  \\
		& VRF   & 1     & 4     & 69    & 33    & 38558  \\
		& Time  &        1.5  &        8.9  &        9.0  &        9.8  &        10.1  \\\hline
		\multirow{3}[0]{*}{Ex. 3} & Est.  & 11.010421  & 10.982053  & 10.978910  & 10.985299  & 10.984770  \\
		& VRF   & 1     & 11    & 15    & 38    & 112  \\
		& Time  &        1.6  &        9.3  &        9.4  &       10.1  &        10.5  \\\hline
		\multicolumn{7}{c}{$d=128$} \\
		\multirow{3}[0]{*}{Ex. 1} & Est.  & 0.485271  & 0.485403  & 0.484720  & 0.485013  & 0.484814  \\
		& VRF   & 1     & 3     & 83    & 4     & 974  \\
		& Time  & 14.3  & 72.5  & 77.0  & 75.2  & 82.4  \\\hline
		\multirow{3}[0]{*}{Ex. 2} & Est.  & 0.563015  & 0.563102  & 0.562506  & 0.562707  & 0.562602  \\
		& VRF   & 1     & 4     & 120   & 6     & 1308  \\
		& Time  &       14.3  &       73.6  &       79.6  &       77.5  &        85.9  \\\hline
		\multirow{3}[0]{*}{Ex. 3} & Est.  & 9.786648  & 9.786195  & 9.827942  & 9.794712  & 9.814580  \\
		& VRF   & 1     & 3     & 8     & 4     & 15  \\
		& Time  &       14.7  &       73.3  &       78.8  &       76.1  &        84.1  \\\bottomrule
	\end{tabular}%
	\label{tab:BS}%
	\begin{tablenotes}
		\item \footnotesize \emph{Note}. The CPU time is reported in milliseconds. ``Ex." and ``Est." stand for ``Example" and ``Estimate", respectively. The results are based on a sample
		size $N=4096$ and estimated with  $100$ replications.
	\end{tablenotes}
\end{table}%

\begin{table}[htbp]
	\centering
	\caption{Effective dimension-related characteristics under the Black-Scholes framework\label{tab:efBS}}\vspace{0.3cm}
	\begin{tabular}{crrrrrrrrrrrr}
		\toprule
		\multirow{2}[0]{*}{Cases} & \multicolumn{1}{c}{\multirow{2}[0]{*}{$d$}} & \multicolumn{5}{c}{sQMC-I}            &       & \multicolumn{5}{c}{sQMC-II} \\\cline{3-7}\cline{9-13}
		&       & $R_{\{1\}}$ & $R_{\{1,2\}}$ & $R_{(1)}$ &     $d_t$ &  $d_{\mathrm{ms}}$ &       & $R_{\{1\}}$ & $R_{\{1,2\}}$ & $R_{(1)}$ &     $d_t$ &  $d_{\mathrm{ms}}$ \\\hline
		\multirow{2}[0]{*}{Ex. 1} & 16    & 15.51 & 29.24 & 84.99 & 13    & 1.37  &       & 99.94 & 99.94 & 99.94 & 1     & 1.00 \\
		& 128   & 1.71  & 3.46  & 73.26 & 109   & 2.99  &       & 99.87 & 99.87 & 99.85 & 1     & 1.00 \\\hline
		\multirow{2}[0]{*}{Ex. 2} & 16    & 0.57  & 16.71 & 89.47 & 14    & 1.24  &       & 0.57  & 99.94 & 99.39 & 2     & 1.01 \\
		& 128   & 0.07  & 1.94  & 81.01 & 109   & 2.38  &       & 0.07  & 99.89 & 99.81 & 2     & 1.00 \\\hline
		\multirow{2}[0]{*}{Ex. 3} & 16    & 1.68  & 10.05 & 63.02 & 16    & 1.47  &       & 1.68  & 91.92 & 91.70 & 15    & 1.12 \\
		& 128   & 0.18  & 1.44  & 52.85 & 127   & 1.98  &       & 0.18  & 81.47 & 84.59 & 119   & 1.46 \\
		\bottomrule
	\end{tabular}%
	\begin{tablenotes}
		\item \footnotesize \emph{Note}. $R_{\{1\}}$, $R_{\{1,2\}}$ and $R_{(1)}$ are reported in percentage. 
	\end{tablenotes}
\end{table}%
Smoothing methods change discontinuous integrands to smooth ones. Smoothness is important to QMC, but it is not the only factor that affects the efficiency of QMC. For  smoothed integrands, effective dimension may have a major impact on the accuracy of QMC. For smooth functions with high effective dimension (such as in the case of the sQMC-I method), one cannot expect a superior convergence rate to appear at moderate sample sizes. Indeed, \cite{owen:1998} showed that for a very smooth function with fully mean dimension, the improvement of QMC over MC  might not set in, until the sample size is large enough.
Therefore, both the smoothing method and the dimension reduction method are important for QMC. The joint effect of  the VPO smoothing method and the MQR method makes the proposed two-step procedure very attractive.

\subsection{Exponential NIG L\'{e}vy Model}
We study the dynamic of the asset $S(t)$ that follows the exponential NIG L\'{e}vy process
\begin{equation*}\label{eq:expNIG}
S(t)=S_0\exp\left\lbrace  L(t) \right\rbrace,
\end{equation*}
where $\{L(t),{t\ge 0}\}$ is the NIG  L\'{e}vy process \citep[see][]{tank:2003}. The NIG L\'{e}vy process is a special GH L\'{e}vy process whose marginal distribution follows the NIG distribution. The NIG distribution $\mathrm{NIG}(\alpha,\beta,\mu,\delta)$ is specified by four parameters, where $\mu$ is the location, $\beta$ indicates the skewness, $\delta$ measures the scale, and $\alpha$ controls the steepness and also affects the tail behavior. The density function of the NIG distribution is given by 
\begin{equation}\label{eq:NIGdensity}
f_{\mathrm{NIG}}(x;\alpha,\beta,\mu,\delta)=\frac{\alpha\delta}{\pi}\exp\left(\delta\sqrt{\alpha^2-\beta^2} + \beta\left( x-\mu \right)  \right) \frac{K_1\left( \alpha s(x)\right) }{s(x)},
\end{equation}
where $x,\mu\in \mathbb{R}$, $0\le |\beta| \le \alpha$,  $\delta>0$, $K_1(x)$ denotes the modified Bessel function of the third kind of order $1$, and 
\begin{equation*}
s(x)=\sqrt{\delta^2+(x-\mu)^2}.
\end{equation*}
The moment generating function of the NIG distribution is given by
\begin{equation}\label{eq:mgfNIG}
M_{\mathrm{NIG}}(u)=\exp\left( \delta\sqrt{\alpha^2-\beta^2} - \delta\sqrt{\alpha^2-(\beta+u)^2} +\mu u \right).
\end{equation}
Note that NIG distribution is closed under convolution, i.e.,
\begin{equation*}\label{convolution}
f_{\mathrm{NIG}}(x;\alpha,\beta,\mu_1,\delta_1)\ast f_{\mathrm{NIG}}(x;\alpha,\beta,\mu_2,\delta_2)=f_{\mathrm{NIG}}(x;\alpha,\beta,\mu_1+\mu_2,\delta_1+\delta_2).
\end{equation*}

%Let $S(t)$ denotes the price of the underlying asset at time $t$. 
As pointed out in \cite{tank:2003} among others, there exists multiple martingale measures that lead to different no-arbitrage prices of a financial derivative.
Following \cite{imai:tan:2009}, we will price options under an equivalent martingale measure determined by the Esscher transform of \cite{gerb:shiu:1994}. Under the Esscher equivalent martingale measure, the NIG distribution turns out to be  $\mathrm{NIG}(\alpha,\beta+\theta,\mu,\delta)$, where the parameter $\theta$ is a solution of the following equation:
\begin{equation}\label{eq:esscher}
r=\log \frac{M_{\mathrm{NIG}}(\theta+1)}{M_{\mathrm{NIG}}(\theta)},
\end{equation}
and $r$ is the riskless rate. The solution of Equation~\eqref{eq:esscher} can be obtained explicitly by  using \eqref{eq:mgfNIG}. One the parameter $\theta$ is identified, one can simulate the asset prices via $$S_i=S_0\exp(x_1+\dots+x_i),$$
where $$x_i=L(t_i)-L(t_{i-1})\simiid \mathrm{NIG}(\alpha,\beta+\theta,\mu\Delta t,\delta\Delta t),$$ and $\Delta t =T/m$.
Note that the closed form of the inverse CDF of the NIG distribution is not available. To circumvent this, we resort to  the numerical inversion method proposed by \cite{horm:2003}; see \cite{imai:tan:2009} for the details on simulating general GH models with QMC. We emphasize that the numerical inversion approximation procedure only takes place in the initial setup. The setup time and the precision of the numerical inversion method were reported and discussed in \cite{imai:tan:2009}.

In the numerical experiments,  the parameters of the NIG distribution are  
\begin{equation}\label{eq:pars}
\alpha = 105.96,\  \beta = -26.15,\  \mu= 360\times 0.00348,\  \delta= 360\times 0.0112,
\end{equation}
which are excerpted from \cite{eber:prau:1998} with annualized $\mu$ and $\delta$ (assume there are 360 days in one year). The four parameters were estimated from the daily returns of the DAX for the period December 15, 1993 to November 26, 1997. We choose the riskless interest rate $r=0.04$ again. These parameters were also chosen in \cite{imai:tan:2009}.
The parameter required for the Esscher transform is then $\theta\doteq-4.87$ by solving Equation \eqref{eq:esscher}. Figure~\ref{fig:density} compares the estimated kernel smoothed density of one million i.i.d. samples of $\mathrm{NIG}(\alpha,\beta+\theta,\mu,\delta)$ generated by the numerical inversion method of \cite{horm:2003} with the true density of $\mathrm{NIG}(\alpha,\beta+\theta,\mu,\delta)$. We observe that the two densities are almost the same. This clearly supports the high accuracy of the numerical inversion method.

\begin{figure}[ht]
	\includegraphics[width=\hsize]{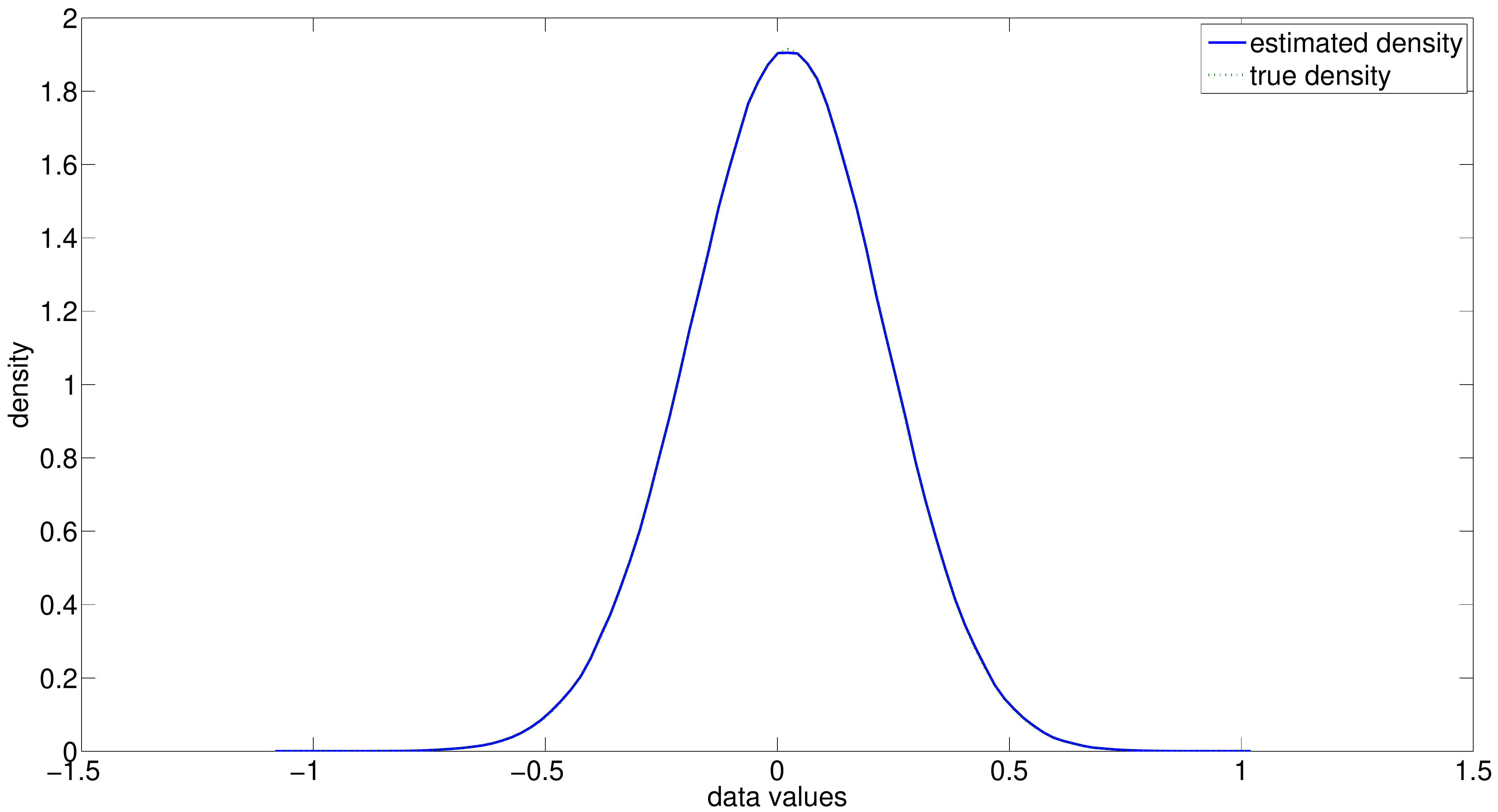}
	\caption{The dashed line plots the true density of  $\mathrm{NIG}(\alpha,\beta+\theta,\mu,\delta)$ using \eqref{eq:NIGdensity}, where the parameters are given in \eqref{eq:pars}. The solid line plots the estimated kernel smoothed density of one million i.i.d. samples of $\mathrm{NIG}(\alpha,\beta+\theta,\mu,\delta)$ generated by the numerical inversion method of \cite{horm:2003}. The two curves are almost the same so that it is hard to distinguish the difference. \label{fig:density}}
\end{figure}

We perform numerical experiments for pricing binary options (Example \ref{ex:1}) and down-and-out barrier options (Example \ref{ex:3}) and estimating the delta of the arithmetic Asian call option by the pathwise method (Example \ref{ex:2}) to illustrate   the effectiveness of the proposed method under the exponential NIG model.  We use the first-order Taylor approximations of the asset prices $S_i$ to get matching functions  for the QR and MQR 　methods. Table~\ref{tab:nig} compares the variance reduction factors for the four QMC-based methods. The effects of our proposed method on effective dimension reduction can be assessed by the characteristics reported in Table~\ref{tab:efNIG}.  Again, we observe that the  sQMC-II method outperforms all the other QMC-based methods under the exponential NIG framework. The advantage of QMC methods declines when the dimension becomes large. The gain  is moderate for Example \ref{ex:3} (pricing down-and-out barrier options), similar to the Black-Scholes and the Heston models.  As claimed in the main paper, the payoff of the barrier option \eqref{gamma} involves $m$ discontinuity structures, which is the most difficult function in our examples.   Our proposed method can also be applied to other exponential GH models in a similar way, such as the exponential hyperbolic model.

\begin{table}[htbp]
	\centering
	\caption{Variance reduction factors for the exponential NIG model}
	\begin{tabular}{ccrrrr}
		\toprule
		$m$     & Cases  & \multicolumn{1}{l}{QMC-I} & \multicolumn{1}{l}{QMC-II} & \multicolumn{1}{l}{sQMC-I} & \multicolumn{1}{l}{sQMC-II} \\\hline
		\multirow{3}[0]{*}{16} & Ex. 1 & 5     & 79    & 42    & 135625  \\
		& Ex. 2 & 7     & 97    & 51    & 145368  \\
		& Ex. 3 & 19    & 24    & 45    & 206   \\\hline
		\multirow{3}[0]{*}{64} & Ex. 1 & 4     & 31    & 6     & 927  \\
		& Ex. 2 & 5     & 38    & 8     & 1116  \\
		& Ex. 3 & 3     & 22    & 3     & 30  \\\bottomrule
	\end{tabular}
	\begin{tablenotes}
		\item \footnotesize \emph{Note}.  The parameters for the NIG distribution are given in \eqref{eq:pars}, and the remaining parameters are $S_0=100$, $r=0.04$, $K=100$, $\kappa=90$, $T=1$. The results are based on a sample
		size $N=2^{14}$ and  $100$ replications. ``Ex."  stands for ``Example". \label{tab:nig}
	\end{tablenotes}
	
\end{table}

\begin{table}
	\centering
	\caption{Effective dimension-related characteristics under the exponential NIG framework\label{tab:efNIG}}\vspace{0.3cm}
	\begin{tabular}{rrrrrrrrrrrr}
		\toprule
		\multicolumn{1}{c}{\multirow{2}[0]{*}{Cases}} 	& \multicolumn{5}{c}{sQMC-I}            & \multicolumn{5}{c}{sQMC-II} \\\cline{2-6}\cline{8-12}
		& $R_{\{1\}}$ & $R_{\{1,2\}}$ & $R_{(1)}$ &     $d_t$ &  $d_{\mathrm{ms}}$ & & $R_{\{1\}}$ & $R_{\{1,2\}}$ & $R_{(1)}$ &     $d_t$ &  $d_{\mathrm{ms}}$\\\hline
		
		Ex. 1 & 15.43 & 29.1  & 84.7  & 13    & 1.38  &       & 99.80 & 99.84 & 99.90 & 1     & 1.00 \\
		Ex. 2 & 0.32  & 16.3  & 88.4  & 14    & 1.28  &       & 0.32  & 99.82 & 99.65 & 2     & 1.00 \\
		Ex. 3 & 2.39  & 9.7   & 69.9  & 16    & 1.38  &       & 2.37  & 95.20 & 94.07 & 14    & 1.08 \\
		\bottomrule
	\end{tabular}%
	\begin{tablenotes}
		\item \footnotesize \emph{Note}. $R_{\{1\}}$, $R_{\{1,2\}}$ and $R_{(1)}$ are reported in percentage, $m=d = 16$.
	\end{tablenotes}
\end{table}

\section{Extension to Heston Model}\label{sec:hestonModel}
Under the Heston framework \citep[][]{heston:1993}, the risk-neutral dynamics of the asset can be expressed as
\begin{align*}
\frac{\mrd S(t)}{S(t)} &= rdt+\sqrt{V(t)} \mrd W_1(t),\\
\mrd V(t) &= (\theta-V(t))\nu \mrd t+\sigma \sqrt{V(t)} \mrd W_2(t),
\end{align*} 
where $\nu$ is the mean-reversion parameter of the volatility process $V(t)$, $\theta$ is the long run average price variance, $\sigma$ is the volatility of the volatility, and $W_1(t)$ and $W_2(t)$ are two standard Brownian motions with an instantaneous correlation $\rho$, i.e., $\mathrm{Cov}(W_1(s),W_2(t))=\rho\min(s,t),$ for any $s,t>0$. 
One may write that $W_1(t)=\hat{\rho}B_1(t)+\rho B_2(t)$ and $W_2(t)=B_2(t)$,
where $\hat{\rho}=\sqrt{1-\rho^2}$, $B_1(t)$ and $B_2(t)$ are two independent standard Brownian motions.

We use the Euler-Maruyama scheme to discretize the asset paths in log-space \citep[][]{acht:2013}, resulting in
\begin{equation}\label{eq:discHeston}
\begin{split}
\log(S_i) &= \log(S_{i-1})+(r-V_{i-1}/2)\Delta t+\sqrt{V_{i-1}}\sqrt{\Delta t}(\hat{\rho}z_i^1+\rho z_i^2) ,\\
V_i &= V_{i-1}+(\theta-V_{i-1})\nu \Delta t+\sigma \sqrt{V_{i-1}} \sqrt{\Delta t}z_i^2,
\end{split}
\end{equation} 
where $z_i^1$ and $z_i^2$ ($i=1,\dots,m$) are independent standard normals,  $V_i$ represents the approximation of $V(t_i)$ for $i=1,\dots,m$ and $V_0$ is the initial value of the volatility process. We now let $\bm{z}=(z_1^1, z_1^2,z_2^1, z_2^2,\dots,z_m^1, z_m^2)^\top\sim N(\bm{0},\bm{I}_{2m}).$  
From \eqref{eq:discHeston}, we have
\begin{equation*}\label{eq:hestonEqs}
\begin{split}
S_i &= S_0\exp\left\lbrace  ri\Delta t-\Delta t\sum_{j=0}^{i-1}V_j/2+\sqrt{\Delta t}\sum_{j=0}^{i-1}\sqrt{V_j} \hat{\rho}z_{2j+1}+\rho z_{2j+2})\right\rbrace, \\
V_i &=V_0+ \theta \nu i\Delta t-\nu\Delta t \sum_{j=0}^{i-1}V_j+\sigma \sqrt{\Delta t}\sum_{j=0}^{i-1}\sqrt{V_j} z_{2j+2}.
\end{split}
\end{equation*}
Note that all $V_i$ do not depend on $z_1$, we can therefore rewrite $S_i$ as
\begin{equation}\label{eq:newSAHeston}  
S_i = \exp\left(\sqrt{(1-\rho^2)V_0\Delta t}z_1\right)\zeta_i(\bm{z}_{2{:}d}),
\end{equation}
where the dimension $d=2m$ and
\begin{equation*}
\zeta_i(\bm{z}_{2{:}d}) = S_0\exp\left\lbrace ri\Delta t-\Delta t\sum_{j=0}^{i-1}V_j/2+\rho\sqrt{\Delta t V_0} z_2+\sqrt{\Delta t}\sum_{j=1}^{i-1}\sqrt{V_j} (\hat{\rho}z_{2j+1}+\rho x_{2j+2})\right\rbrace.
\end{equation*}

Here we only verify the variable separation condition for  the average function $q_{\mathrm{A}}(\bm S)=S_{\mathrm{A}}-\kappa$. The analysis for the component function and the extreme function is similar. Using $\bm{z}=\Phi^{-1}(\bm u)$, it follows from  \eqref{eq:newSAHeston} that
\begin{align*}
S_{\mathrm{A}} =\frac 1m\sum_{i=1}^m S_i=\frac{\exp\left(\sqrt{(1-\rho^2)V_0\Delta t}\Phi^{-1}(u_1)\right)}{m}\sum_{i=1}^m \zeta_i(\Phi^{-1}(\bm{u}_{2{:}d})).
\end{align*}
We thus have the following equivalence
\begin{equation*}
\{q_{\mathrm{A}}(\bm S)>0\}\Leftrightarrow\{S_{\mathrm{A}}-\kappa>0\}\Leftrightarrow
\{\hat{\gamma}(\bm{u}_{2{:}d})<u_1<1\},
\end{equation*}
where
\begin{equation*}
\hat{\gamma}(\bm{u}_{2{:}d}):=\Phi\left(\frac{1}{\sqrt{(1-\rho^2)V_0\Delta t}}\left(\log(\kappa  m)-\log\left(\sum_{i=1}^m \zeta_i(\Phi^{-1}(\bm{u}_{2{:}d}))\right)\right)\right).
\end{equation*}
This implies that the variable
separation condition is satisfied for the average function. The representation \eqref{eq:newSAHeston} guarantees the applicability of the MQR method, since the variable
separation condition still holds if replacing $\bm{z}$ with $\bm{Uz}$ in \eqref{eq:newSAHeston}, where $\bm U$ has the form \eqref{UMatrix}.

We perform numerical experiments for Examples \ref{ex:1}--\ref{ex:3} to illustrate the flexibility and the effectiveness of the proposed method under the Heston framework.  We use the first-order Taylor approximations of the asset prices $S_i$ to get matching functions for the QR and MQR methods.

In our experiments, we choose $
	S_0=100,\ V_0=\theta=\sigma=0.2,\ T=1,\ \nu=1,K=100,\ m\in\{16,64\},$ and $\rho=\pm 0.5$.
Again, the nominal dimension is $d=2m\in\{32,128\}$. Table~\ref{tab:numerHeston} presents the comparison of VRFs for Examples~\ref{ex:1} -- \ref{ex:3}. Note that the computational costs of the QMC methods are quite close as in the Black-Scholes model. So it is fair to take the VRF as a measure of the quality of the QMC methods. Table~\ref{tab:efHeston} presents the effective dimension-related characteristics for $d=32$ and $\rho = 0.5$. 
For both examples, the sQMC-II method
	can further improve the efficiency of the QMC-II method by a large factor. Table~\ref{tab:efHeston} shows that the sQMC-II method yields the larger degree of additivity $R_{(1)}$, the smaller truncation dimension $d_t$ and the mean dimension
	$d_{\mathrm{ms}}$.  These demonstrates the good effect of the proposed method. Similar to the Black-Scholes model, we get  smaller VRFs for Example~\ref{ex:3}. 

\begin{table}[htbp]
	\centering
	\caption{Variance reduction factors for the Heston model}
    \vspace{0.3cm}
	\begin{tabular}{rrrrrr}
		\toprule
		\multicolumn{1}{c}{Cases} & \multicolumn{1}{c}{$\rho$} & \multicolumn{1}{l}{QMC-I} & \multicolumn{1}{l}{QMC-II} & \multicolumn{1}{l}{sQMC-I} & \multicolumn{1}{l}{sQMC-II} \\\hline
		\multicolumn{6}{c}{$m=16$} \\
		\multirow{2}[0]{*}{Ex. 1} & 0.5   & 3     & 23    & 11    & 1187  \\
		& $-0.5$  & 3     & 26    & 10    & 754  \\\hline
		\multirow{2}[0]{*}{Ex. 2} & 0.5   & 5     & 39    & 21    & 1418  \\
		& $-0.5$  & 6     & 45    & 19    & 863  \\\hline
		\multirow{2}[0]{*}{Ex. 3} & 0.5   & 4     & 7     & 8     & 33  \\
		& $-0.5$  & 7     & 14    & 19    & 93  \\\hline
		\multicolumn{6}{c}{$m=64$} \\
		\multirow{2}[0]{*}{Ex. 1} & 0.5   & 3     & 23    & 5     & 149  \\
		& $-0.5$  & 3     & 18    & 5     & 105  \\\hline
		\multirow{2}[0]{*}{Ex. 2} & 0.5   & 4     & 37    & 7     & 195  \\
		& $-0.5$  & 5     & 29    & 8     & 144  \\\hline
		\multirow{2}[0]{*}{Ex. 3} & 0.5   & 2     & 5     & 2     & 15  \\
		& $-0.5$  & 3     & 7     & 4     & 14  \\
		
		\bottomrule
	\end{tabular}
	\label{tab:numerHeston}%
\end{table}%
\begin{table}[htbp]
	\centering
	\caption{Effective dimension-related characteristics under the Heston framework\label{tab:efHeston}}\vspace{0.3cm}
	\begin{tabular}{rrrrrrrrrrrr}
		\toprule
		\multicolumn{1}{c}{\multirow{2}[0]{*}{Cases}} 	& \multicolumn{5}{c}{sQMC-I}            & \multicolumn{5}{c}{sQMC-II} \\\cline{2-6}\cline{8-12}
		& $R_{\{1\}}$ & $R_{\{1,2\}}$ & $R_{(1)}$ &     $d_t$ &  $d_{\mathrm{ms}}$ & & $R_{\{1\}}$ & $R_{\{1,2\}}$ & $R_{(1)}$ &     $d_t$ &  $d_{\mathrm{ms}}$\\\hline
		Ex. 1 & 3.89  & 14.8  & 81.7  & 26    & 1.54  &       & 99.14 & 99.16 & 99.20 & 1     & 1.01 \\\hline
		Ex. 2 & 0.72  & 4.9   & 86.4  & 27    & 1.30  &       & 0.71  & 99.14 & 98.31 & 2     & 1.02 \\\hline
		Ex. 3 & 0.70  & 3.3   & 47.6  & 32    & 1.74  &       & 0.70  & 75.79 & 79.28 & 32    & 1.28 \\
		\bottomrule
	\end{tabular}%
	\begin{tablenotes}
		\item \footnotesize \emph{Note}. $R_{\{1\}}$, $R_{\{1,2\}}$ and $R_{(1)}$ are reported in percentage, $\rho = 0.5$ and $d=2m=32$.
	\end{tablenotes}
\end{table}%

\section{Conclusions}\label{SectionCon}
High dimensionality and discontinuities are challenges for QMC since they can dramatically degrade the
performance of  QMC. Developing methods to deal with both the high dimensionality and discontinuities is of considerable practical importance. Dealing only with a single aspect may lead to unsatisfactory results. Some PGMs have been proposed to realign discontinuity structures, resulting in QMC-friendly discontinuities, but discontinuities are still presented in the resulting function. It is known that QMC could possess a superior asymptotic convergence rate for smooth functions. Motivated by this, we developed the VPO  method aiming at removing the discontinuities completely and thus improving the smoothness of the functions. 

However, smoothness is not the only factor that affects the performance of QMC. The effective dimensions of the integrands can be large if we use PGMs naively. We therefore proposed the MQR method that can be compatible with the
smoothing method and has the potential to reduce the effective dimension by concentrating the variance on the first few variables. Numerical experiments showed that the MQR method in combination with the VPO smoothing method provides a consistent advantage over other methods.
This combination leads to the smallest effective dimension in several numerical examples, explaining the superiority of the combined procedure.

\section*{Appendix: Effective Dimension}
The effective dimension is an important measure of the complexity of QMC integration for smooth functions. Any square integrable function $h(\bm{u})$ on $(0,1)^d$ has an ANOVA decomposition \citep[][]{efron:stein:1981} as a sum
$h(\bm{u}) = \sum_{v\subseteq \{1,\dots,d\}}h_v(\bm{u}),$
where $h_v(\bm{u})$ depends on $\bm{u}$  through $u_j$ for $j\in v$. The ANOVA terms are defined recursively by
$h_v(\bm{u})=\int_{(0,1)^{d-\abs{v}}} h(\bm{u})\mrd \bm{u}_{-v}-\sum_{w\subsetneq v}h_w(\bm{u}),$
where $\abs{v}$ is the cardinality of $v$ and $\bm{u}_{-v}$ denotes the vector of the coordinates of $\bm{u}$ with indices not in $v$.
When $v=\varnothing$, we use the convention $h_{\varnothing}(\bm{u})=\int_{(0,1)^d}h(\bm{u})\mrd \bm{u}$. It is easy to see that the ANOVA
decomposition is orthogonal: $\int h_v(\bm{u})h_w(\bm{u})\mrd \bm{u}=0$ whenever $v\neq w$. This property ensures that
$\sigma^2(h)=\sum_{v\subseteq \{1,\dots,d\}}\sigma^2(h_v),$
where $\sigma^2(h)$ and $\sigma^2(h_v)$ are the variances of $h(\bm{u})$ and $h_v(\bm{u})$, respectively. This decomposition of variance serves to define various effective dimension-related characteristics.

Let $\ell\in\{1,\dots,d\}$ and $v\subseteq \{1,\dots,\ell\}$. The superposition variance
ratio captured by all order-$\ell$ ANOVA terms is defined by
$R_{(\ell)}(h)=\frac1{\sigma^2(h)}\sum_{\abs{v}=\ell} \sigma^2(h_v).$ The truncation variance ratio captured by all the ANOVA terms $h_v$
is defined as
$R_{\{1,\dots,\ell\}}(h)=\sum_{v\subseteq \{1,\dots,\ell\}} \sigma^2(h_v)/\sigma^2(h).$ 
The two variance ratios defined above offer two different ways of measuring  the importance of variables to the function $h$.
If $R_{\{1,\dots,\ell\}}(h)\approx 1$, we may say that $h$ depends mainly on the first $\ell$ variables.
The superposition variance ratio of the first order $R_{(1)}$ measures the degree of additivity of $h$.
If $R_{(1)}(h)=1$, then the function $h$ is additive; if $R_{(1)}(h)\approx 1$, then $h$ is highly additive.

The concepts of effective dimension are proposed by \cite{Caflisch1997}. Let $p$ be
a parameter close to $1$ (we
choose $p =0.99$). The effective dimension
in the truncation sense of $h$ is the smallest integer
$d_t$ such that $R_{\{1,\dots,d_t\}}(h)\geq p$. The effective dimension
in the superposition sense of $h$ is the smallest integer
$d_s$ such that $\sum_{\ell=1}^{d_s}R_{(\ell)}(h)\geq p$. 
The 
mean dimension proposed by \cite{owen:2003} is defined
as $d_{\mathrm{ms}}=\sum_{\ell=1}^d\ell R_{(\ell)}$. 
Some effective
dimension-related characteristics  can be computed numerically \citep[][]{kuch:2011,liu:owen:2006, sobol:2005, sobol2005b, Wang2003}.

\section*{Acknowledgments}
Zhijian He was supported by the National Science Foundation of
China under grant 71601189.
Xiaoqun Wang was supported by the National Science Foundation of
China under
grant 71471100.

\end{document}